\documentclass[11pt,dvipdfmx]{amsart}

\usepackage{amsmath,amssymb}
\usepackage{amsthm}
\usepackage[abbrev]{amsrefs}
\usepackage{latexsym}
\usepackage{txfonts}
\usepackage{graphicx}
\usepackage{tikz}
\usepackage{bm}
\usepackage{shuffle}

\allowdisplaybreaks[1]

\title[$p$-adic higher Mahler measures]{A note on $p$-adic higher Mahler measures}
\date{}
\author{Yu Katagiri}

\newtheorem{thm}{Theorem}[section]
\newtheorem*{thm*}{Theorem}
\newtheorem{lem}[thm]{Lemma}
\newtheorem*{lem*}{Lemma}
\newtheorem{prop}[thm]{Proposition}
\newtheorem*{prop*}{Proposition}

\newtheorem{cor*}{Corollary}

\theoremstyle{definition}
\newtheorem{dfn}[thm]{Definition}
\newtheorem*{dfn*}{Definition}
\newtheorem{ex}[thm]{Example}
\newtheorem*{ex*}{Example}
\newtheorem{rmk}[thm]{Remark}
\newtheorem*{rmk*}{Remark}

\makeatletter
\@addtoreset{equation}{section}

\makeatother

\makeatletter\@namedef{subjclassname@2020}{\textup{2020} Mathematics Subject Classification}\makeatother


\begin{document}
\maketitle

\begin{abstract}
Kurokawa, Lal\'{\i}n and Ochiai introduced and studied the higher Mahler measures, which are generalization of the classical Mahler measure. In this article, we introduce $p$-adic higher Mahler measures and prove $p$-adic analogues of Akatsuka's results.
\end{abstract}

\section{Introduction}

Let $|\cdot |_\infty$ be the usual absolute value on $\mathbb{C}$. For a non-zero Laurent polynomial $f \in \mathbb{C}[t_1^{\pm 1}, \dots ,t_n^{\pm 1}]$ of $n$ variables, the {\it Mahler measure} of $f$ is given by
\begin{align*}
 m(f) \coloneqq& \frac{1}{(2 \pi \sqrt{-1})^n} \int_{T^n} \log |f(z_1, \dots ,z_n)|_\infty \frac{dz_1}{z_1} \cdots \frac{dz_n}{z_n} \in \mathbb{R},
\end{align*}
where
\begin{align*}
T^n=\{(z_1, \dots ,z_n) \in \mathbb{C}^n \mid |z_1|_\infty= \dots =|z_n|_\infty=1\}
\end{align*}
is the $n$-torus. This integral in  fact exists even if $f$ has zeroes on $T^n$. 
It is known that Mahler measures appear in several areas and relate with various objects, for example, entropies in dynamical systems, and special values of $L$-functions in number theory (see \cite{LW88}, \cite{LSW90}, \cite{Sm91}, \cite{RV97} and \cite{De97} for instance).

Kurokawa, Lal\'{\i}n and Ochiai considered the following generalization \cite[Definition 1]{KLO08}: for a non-zero Laurent polynomial  $f \in \mathbb{C}[t_1^{\pm 1}, \dots ,t_n^{\pm 1}]$ and a non-negative integer $k$, we define
\begin{align*}
 m_k(f) \coloneqq& \frac{1}{(2 \pi \sqrt{-1})^n} \int_{T^n} \log^k |f(z_1, \dots ,z_n)|_\infty \frac{dz_1}{z_1} \cdots \frac{dz_n}{z_n} \in \mathbb{R}
\end{align*}
and call $m_k(f)$ the {\it $k$-higher Mahler measure} of $f$. Note that $m_1(f)=m(f)$ and $m_0(f)=1$. Furthermore, Akatsuka introduced the {\it zeta Mahler measure} of $f$ to be a holomorphic function
\begin{align*}
Z(s, f) \coloneqq \frac{1}{(2 \pi \sqrt{-1})^n} \int_{T^n}  |f(z_1, \dots ,z_n)|^s_\infty \frac{dz_1}{z_1} \cdots \frac{dz_n}{z_n},
\end{align*}
where $s$ is a complex variable \cite[Section 1]{Ak09}. 
This integral converges absolutely in ${\rm Re}(s)>\sigma_0(f)$ for some $\sigma_0(f) \leq 0$ \cite[Proposition 2.1]{Ak09} and $Z(s, f)$ has the Taylor coefficients $m_k(f)$, i.e., 
\begin{align*}
Z(s, f)=\sum_{k=0}^\infty \frac{m_k(f)}{k!}s^k.
\end{align*}
%


The following two examples are known in previous works.

\begin{thm}[{\cite[Theorem 4(1)]{Ak09}}]\label{Ak1}
Let $c \in \mathbb{R}$ with $c>2$ and $f(t)=t^2-ct+1$. Then, for any $s \in \mathbb{C}$, we have
\begin{align*}
Z(s, f)=\alpha_c^s {}_2F_1\left(-s, -s ; 1 ; \beta_c^2\right),
\end{align*}
where $\alpha_c$ and $\beta_c$ are the roots of $f(t)=0$ with $0<\beta_c<1<\alpha_c$, and the generalized hypergeometric series ${}_{r+1}F_r$ is given by
\begin{align*}
{}_{r+1}F_r \left( a_1, \cdots, a_{r+1} ; b_1, \cdots, b_r ; z\right)=\sum_{n=0}^\infty \frac{(a_1)_n \cdots (a_{r+1})_n}{(b_1)_n \cdots (b_{r})_n n!}z^n
\end{align*}
with the Pochhammer symbol $(\cdot)_n$ defined by $(a)_0=1$ and $(a)_n=a(a+1)\cdots (a+n-1)$ for $n \geq 1$.
\end{thm}

\begin{rmk}\label{Ak comp}
Let $c \in \mathbb{R}$ with $c>2$ and $f(t)=t^2-ct+1$. Since 
\begin{align*}
\frac{d^kZ}{ds^k}(0, f)=m_k(f)
\end{align*}
for any $k \geq 0$, we can compute
\begin{align*}
m_2(f)&=\log^2 \alpha_c+2\operatorname{Li}_2(\beta_c^2), \\
m_3(f)&=\log^3 \alpha_c+6(\log \alpha_c)\operatorname{Li}_2(\beta_c^2)-12\operatorname{Li}_{(1, 2)}(\beta_c^2)
\end{align*}
by Theorem \ref{Ak1}. Here, for a tuple $\textbf{k}=(k_1, \cdots , k_r) \in \mathbb{Z}_{>0}^r$, the {\it multiple polylogarithm} $\operatorname{Li}_{\textbf{k}}(t)$ is given by
\begin{align}\label{def multipolylog}
\operatorname{Li}_{\textbf{k}}(t)=\sum_{0<m_1<\cdots<m_r}\frac{t^{m_r}}{m_1^{k_1}\cdots m_r^{k_r}}
\end{align}
and $\operatorname{Li}_{\textbf{k}}(t)$ is absolutely convergent on $|t|_\infty<1$.
\end{rmk}

\begin{thm}[{\cite[Theorem 17]{KLO08}}, {\cite[Theorem 6]{Ak09}}]\label{Ak2}
Let $c \in \mathbb{R}$ with $c>4$ and $f(t_1, t_2)=t_1+t_1^{-1}+t_2+t_2^{-1}+c$. Then we have
\begin{align*}
Z(s, f)=c^s {}_3F_2\left(\frac{1}{2}, -\frac{s}{2}, \frac{1-s}{2} ; 1, 1 ; \frac{16}{c^2}\right).
\end{align*}
\end{thm}

\begin{rmk}
The Mahler measure of $f(t_1, t_2)=t_1+t_1^{-1}+t_2+t_2^{-1}+c$ was studied by Rodriguez-Villegas \cite{RV97}. He proved
\begin{align}\label{RV}
m(f)=\log c -\frac{2}{c^2} {}_4F_3\left(\frac{3}{2}, \frac{3}{2}, 1, 1 ; 2, 2, 2 ; \frac{16}{c^2}\right).
\end{align}
\end{rmk}

In this article, we introduce a $p$-adic analogue of higher Mahler measures and consider $p$-adic analogues of the above theorems.
Let $p$ be a prime and $\mathbb{C}_p$ be the completion of $\overline{\mathbb{Q}_p}$ with the norm $|\cdot |$ normalized so that $|p|=p^{-1}$. In \cite{BD99}, Besser and Deninger introduced the {\it $p$-adic Mahler measure}, which is a $p$-adic analogue of the Mahler measures, as follows: for $f \in \mathbb{C}_p[t_1^{\pm 1}, \dots ,t_n^{\pm 1}]$ with no zeroes of $p$-adic absolute value one, we define 
\begin{align*}
m_{p}(f) \coloneqq \lim_{\substack{N \to \infty \\ (N,p)=1}} \frac{1}{N^n} \sum_{\zeta \in \mu_N^n} \log_p f(\zeta_1, \cdots, \zeta_n).
\end{align*}
Here, $\mu_N$ is the set of the $N$-th roots of unity in $\mathbb{C}_p$ and $\log_p : \mathbb{C}_p^{\times} \rightarrow \mathbb{C}_p$ is the $p$-adic logarithm normalized so that $\log_p p=0$. (It is known that the multiplicative group of $\mathbb{C}_p$ has the decomposition
\begin{align*}
\mathbb{C}_p^{\times} \simeq p^{\mathbb{Q}}\times \mu_{ur} \times U_0,
\end{align*}
where $\mu_{ur}$ is the set of roots of unity in $\mathbb{C}_p$ whose order in $\mathbb{C}_p^{\times}$ is prime to $p$, and $U_0 =\{z \in \mathbb{C}_p \mid |z-1|<1\}$.
If we denote by $\langle a \rangle$ the corresponding part of $a$ in $U_0$ for $a \in \mathbb{C}_p^{\times}$, then the value $\log_p a$ is given by 
\begin{align*}
\log_p a=\sum_{n=1}^\infty \frac{(-1)^{n-1}}{n}(\langle a\rangle-1)^n.
\end{align*}
See also \cite[Proposition 5.4]{Wa97}.)

Besser and Deninger studied $p$-adic analogues of several known results of the classical Mahler measure. Especially, they proved that the $p$-adic Mahler measure relates to the regulator map into (modified) syntomic cohomology. This is a $p$-adic analogue of the remarkable relation between the Mahler measure and the regulator map into Deligne cohomology, which was proved by Deninger \cite{De97}. (See \cite[Section 2]{BD99} for the detail.) Moreover, Deninger introduced the $p$-adic entropy, which is a $p$-adic analogue of the periodic entropy, and proved that it relates to the $p$-adic Mahler measure. This corresponds to an analogue of \cite[Theorem 3.1]{LSW90} in dynamical systems. (See \cite{De09}, \cite{Br10}, \cite{De12} and \cite{Ka21a} for the theory of the $p$-adic entropy.)

In the same way as Besser-Deninger's definition, we define the {\it $p$-adic $k$-higher Mahler measure} by
\begin{align}\label{deflim m_p}
m_{p, k}(f) \coloneqq \lim_{\substack{N \to \infty \\ (N,p)=1}} \frac{1}{N^n} \sum_{\zeta \in \mu_N^n} \log_p^k f(\zeta_1, \cdots, \zeta_n)
\end{align}
for $k \geq 1$. (See Section 2 for the existence of $m_{p, k}(f)$.) The following is one of our main results.

\begin{thm}\label{main1}
Let $\alpha, \beta \in \mathbb{C}_p$ with $0<|\beta|<1<|\alpha|$. Then, for each $k \geq 1$, we have
\begin{align*}
&m_{p, k}((t-\alpha)(t-\beta)) \\
=&\log_p^k \alpha+\sum_{\substack {i+j \leq k \\ i, j \geq 1}} (-1)^{i+j} \frac{k!}{(k-i-j)!}\left(\log_p^{k-i-j} \alpha\right) \operatorname{Li}_{\left(\{1\}^{i-1}\ast \{1\}^{j-1}, 2\right)} \left(\frac{\beta}{\alpha}\right),
\end{align*}
where $\ast$ is the harmonic product defined in Section 3.
\end{thm}

We note that the radius of the $p$-adic convergence of the multiple polylogarithm series is $1$ (see the proof of \cite[Theorem 1.6]{Ka21b} for example), and hence the statement of Theorem \ref{main1} makes sense.

We further introduce a {\it naive} $p$-adic analogue of the zeta Mahler measure as the formal power series whose Taylor coefficients are $m_{p, k}(f)$, i.e.,
\begin{align*}
Z_p(X, f)=\sum_{k=0}^\infty \frac{m_{p, k}(f)}{k!}X^k \in \mathbb{C}_p[[X]].
\end{align*}
In this article, we call $Z_p(X, f)$ the {\it $p$-adic zeta Mahler measure} of $f$. In Section 4, we verify that $p$-adic zeta Mahler measures converge on the closed unit disc under certain conditions and prove the following theorems, which are $p$-adic analogues of Theorem \ref{Ak1} and Theorem \ref{Ak2}.

\begin{thm}\label{main2}
Let $K$ be a finite extension of $\mathbb{Q}_p$ (inside $\mathbb{C}_p$) of ramified index $e <p-1$ and $\alpha, \beta \in K$ with $0<|\beta|<1<|\alpha|$. Then, for any $s \in \mathbb{C}_p$ with $|s|\leq 1$, we have
\begin{align*}
Z_p(s, (t-\alpha)(t-\beta))=\langle \alpha \rangle^s {}_2F_1\left(-s, -s ; 1 ; \frac{\beta}{\alpha} \right),
\end{align*}
where $\langle a \rangle^s$ is given by
\begin{align*}
\langle a \rangle^s=\sum_{n=0}^\infty \binom{s}{n} (\langle a \rangle-1)^n
\end{align*}
with
\begin{align*}
\binom{s}{n}=
\begin{cases}
\frac{s(s-1)\cdots (s-n+1)}{n!} & n \geq 1, \\
1 & n=0.
\end{cases}
\end{align*}
\end{thm}

\begin{thm}\label{main3}
Let $K$ be a finite extension of $\mathbb{Q}_p$ (inside $\mathbb{C}_p$) of ramified index $e <p-1$ and $c \in K$ with $|c|>1$. Put $f(t_1, t_2)=t_1+t_1^{-1}+t_2+t_2^{-1}+c$. Then, for any $s \in \mathbb{C}_p$ with $|s|\leq 1$, we have
\begin{align*}
Z_p(s, f)=\langle c \rangle^s {}_3F_2\left(\frac{1}{2}, -\frac{s}{2}, \frac{1-s}{2} ; 1, 1 ; \frac{16}{c^2} \right).
\end{align*}
\end{thm}

Kurokawa, Lal\'{\i}n, Ochiai and Akatsuka considered and computed more examples of higher Mahler measures and zeta Mahler measures, containing the case $0 \leq c \leq 2$ of Theorem \ref{Ak1}. However, we may not obtain analogues of these results as long as we use the limit $(\ref{deflim m_p})$ for the definition of the $p$-adic (higher) Mahler measures. This is because the limit $(\ref{deflim m_p})$ does not exist if a Laurent polynomial $f \in \mathbb{C}_p[t_1^{\pm 1}, \dots ,t_n^{\pm 1}]$ has a zero of $p$-adic absolute value one (see the remark written below Proposition 1.3 in \cite{BD99}). This also explains our assumptions about the absolute values of coefficients in our results.

\vspace{10pt}
\noindent
\textsc{Notation:}~ In this paper, let $p$ be a prime and $\mathbb{C}_p$ be the completion of $\overline{\mathbb{Q}_p}$ with the $p$-adic norm $|\cdot |$ normalized by $|p|=p^{-1}$. We denote by $\mathfrak{m}_p$ the maximal ideal of the ring of integers in $\mathbb{C}_p$. For a positive constant $r$, we define 
\begin{align*}
D(r^{-})&=\{ z \in \mathbb{C}_p \mid |z|<r \}, \\
D(r)&=\{ z \in \mathbb{C}_p \mid |z|\leq r \}.
\end{align*}

\vspace{10pt}
\noindent
\textsc{Acknowledgment:}~ Contents in this article are partially based on the author's master thesis in Tohoku University. The author is grateful to his supervisor Professor Takao Yamazaki for his advice and helpful comments. The author would like to thank Hiroyuki Ochiai and Satoshi Kumabe for their helpful discussions and comments. The author also thanks Naho Kawasaki for several comments on the contents in Section 3. 
This work is supported by JSPS Grant-in-Aid for Transformative Research Areas (A) (22H05107).

\section{$p$-adic higher Mahler measures}

In this section, we show the existence of $p$-adic higher Mahler measures and verify some properties. The contents in this section are mainly based on \cite[Section 1]{BD99}.

\begin{dfn}[{\cite[Section 1]{BD99}}]
For a function $f : \mu_{ur}^n \rightarrow \mathbb{C}_p$, if the limit
\begin{align}\label{Sch int}
\int_{T_p^n} f(z_1,\cdots,z_n) \frac{dz_1}{z_1} \cdots \frac{dz_n}{z_n} \coloneqq \lim_{\substack{N \to \infty \\ (N,p)=1}} \frac{1}{N^n} \sum_{\zeta \in \mu_N^n} f(\zeta)
\end{align}
exists, we call it the Shnirelman integral of $f$ and say $f$ is (Shnirelman) integrable.
\end{dfn}

We have the following properties of the Shnirelman integrals.
\begin{lem}[{\cite[Lemma 1.1]{BD99}}]\label{lemBD1}
\begin{enumerate}
\item If the limit in $(\ref{Sch int})$ exists, then we have
\begin{align*}
\left|\int_{T_p^n} f(z_1,\cdots,z_n) \frac{dz_1}{z_1} \cdots \frac{dz_n}{z_n} \right| \leq \sup_{\zeta \in \mu_{ur}^n} \{|f(\zeta)|\}.
\end{align*}
\item Define an affinoid algebra $\mathcal{R}_n(\mathbb{C}_p)$ over $\mathbb{C}_p$ by
\begin{align*}
\mathcal{R}_n(\mathbb{C}_p) \coloneqq \left\{\sum_{v \in \mathbb{Z}^n} a_v t_1^{v_1}\cdots t_n^{v_n} \mid a_v \in \mathbb{C}_p \ \text{with} \ |a_v| \to 0 \ \text{as} \ |v_1|_\infty + \cdots +|v_n|_\infty \to \infty \right\}.
\end{align*}
Then each $f=\sum_{v \in \mathbb{Z}^n} a_v t_1^{v_1}\cdots t_n^{v_n} \in \mathcal{R}(\mathbb{C}_p)$ defines a well-defined map from $\mu_{ur}^n$ to $\mathbb{C}_p$ and we have
\begin{align*}
\int_{T_p^n} f(z_1,\cdots,z_n) \frac{dz_1}{z_1} \cdots \frac{dz_n}{z_n} =a_0.
\end{align*}
\end{enumerate}
\end{lem}

The proofs of following two lemmas are elementary and hence omitted.

\begin{lem}\label{ele lem}
Suppose that a family $\{f_m : \mu_{ur}^n \rightarrow \mathbb{C}_p\}_m$ of functions converges uniformly to a function $f : \mu_{ur}^n \rightarrow \mathbb{C}_p$.
\begin{enumerate}
\item If $f_m$ is integrable for any $m \geq 1$, then the limit 
\begin{align*}
\lim_{m \to \infty} \int_{T_p^n} f_m(z_1,\cdots,z_n) \frac{dz_1}{z_1}\cdots \frac{dz_n}{z_n}
\end{align*}
exists and we have
\begin{align*}
\lim_{m \to \infty} \int_{T_p^n} f_m(z_1,\cdots,z_n) \frac{dz_1}{z_1}\cdots \frac{dz_n}{z_n}=\int_{T_p^n} f(z_1,\cdots,z_n) \frac{dz_1}{z_1}\cdots \frac{dz_n}{z_n}.
\end{align*}
\item If there exists a constant $C>0$ such that $|f_m(z)|\leq C$ for any $m \geq 1$ and $z \in \mu_{ur}^n$, then $\{f_m^j\}_m$ converges uniformly to $f^j$ for each $j \geq 0$.
\end{enumerate}
\end{lem}

\begin{lem}\label{log int lem}
Suppose that a function $f : \mu_{ur}^n \rightarrow \mathbb{C}_p$ is integrable. Then, for any $1 \leq i \leq n$, the function $f(z_1,\cdots,z_n) \log_p z_i$ is integrable and the integral is $0$.
\end{lem}

The following lemma is well known in rigid analysis and plays an important role to consider $p$-adic Mahler measures.

\begin{lem}[{\cite[Lemma 1.2]{BD99}}]
Let $f \in \mathbb{C}_p[t_1^{\pm 1}, \dots ,t_n^{\pm 1}]$. Then $f$ does not vanish in any point of $T_p^n$ if and only if we have
\begin{align}\label{lem 1.2}
f(t_1, \cdots, t_n)=at_1^{l_1}\cdots t_n^{l_n}(1+g(t_1, \cdots, t_n))
\end{align}
for some $a \in \mathbb{C}_p^{\times}, (l_1, \cdots, l_n) \in \mathbb{Z}^n$ and $g \in \mathfrak{m}_p[t_1^{\pm 1}, \dots ,t_n^{\pm 1}]$.
\end{lem}

Combining above lemmas, we can check the existence of $p$-adic higher Mahler measures as follows.

\begin{prop}[cf. {\cite[Proposition 1.3]{BD99}}]\label{exist pHMM}
Suppose that $f \in \mathbb{C}_p[t_1^{\pm 1}, \dots ,t_n^{\pm 1}]$ does not vanish in any point of $T_p^n$ and write $f$ as in $(\ref{lem 1.2})$. Then, for each $k \geq 1$, the Shnirelman integral 
\begin{align*}
m_{p,k}(f) \coloneqq \int_{T_p^n} \log_p^k f(z_1, \cdots, z_n) \frac{dz_1}{z_1}\cdots\frac{dz_n}{z_n}
\end{align*}
exists and is given by
\begin{align*}
m_{p,k}(f)=\sum_{i=0}^k \binom{k}{i} \left(\log_p^i a\right)b_{k-i},
\end{align*}
where $b_j=\lim_{M \to \infty} [\{\sum_{m=1}^M (-1)^{m+1}g(z_1, \cdots, z_n)^m/m\}^j]_0$ and $[\cdot]_0$ means the $0$-th coefficient.
\end{prop}

\begin{proof}
We compute
\begin{align*}
m_{p,k}(f)&=\int_{T_p^n} \left\{\log_p az_1^{l_1}\cdots z_n^{l_n}(1+g(z_1,\cdots,z_n))\right\}^k \frac{dz_1}{z_1}\cdots\frac{dz_n}{z_n} \\
&=\sum_{i=0}^k \binom{k}{i} \int_{T_p^n} \left(\log_p^i az_1^{l_1}\cdots z_n^{l_n}\right)\left(\log_p^{k-i} (1+g(z_1,\cdots,z_n))\right) \frac{dz_1}{z_1}\cdots\frac{dz_n}{z_n}
\end{align*}
and, for each $0 \leq i \leq n$, 
\begin{align*}
  &\int_{T_p^n} \left(\log_p^i az_1^{l_1}\cdots z_n^{l_n}\right)\left(\log_p^{k-i} (1+g(z_1,\cdots,z_n))\right) \frac{dz_1}{z_1}\cdots\frac{dz_n}{z_n}\\
=&\lim_{\substack{N \to \infty \\ (N,p)=1}}\frac{1}{N^n}\sum_{\zeta \in \mu_N}\left\{\sum_{\substack{i_0,\cdots,i_n \geq 0 \\ i_0+\cdots+i_n=i}}\frac{i!}{i_0!\cdots i_n!} \left(\log_p^{i_0} a\right) \prod_{j=1}^n\left(l_j \log_p \zeta_j\right)^{i_j}\right\} \times \\
 &\left(\log_p^{k-i}(1+g(\zeta_1,\cdots,\zeta_n))\right) \\
=&\lim_{\substack{N \to \infty \\ (N,p)=1}}\frac{1}{N^n}\sum_{\zeta \in \mu_N}\left(\log_p^i a\right)\left(\log_p^{k-i}(1+g(\zeta_1,\cdots,\zeta_n))\right) \\
=&\left(\log_p^i a\right)\int_{T_p^n} \log_p^{k-i} (1+g(z_1,\cdots,z_n)) \frac{dz_1}{z_1}\cdots\frac{dz_n}{z_n}.
\end{align*}
Since $\{\sum_{m=1}^M (-1)^{m+1}g(z_1, \cdots, z_n)^m/m\}_{M \geq 1}$ converges uniformly to $\log_p (1+g(z_1,\cdots,z_n))$ on $\mu_{ur}^n$, Lemma \ref{ele lem} implies that
\begin{align*}
&\int_{T_p^n} \log_p^{k-i} (1+g(z_1,\cdots,z_n)) \frac{dz_1}{z_1}\cdots\frac{dz_n}{z_n} \\
=&\lim_{M \to \infty} \int_{T_p^n} \left\{\sum_{m=1}^M \frac{(-1)^{m+1}}{m}g(z_1, \cdots, z_n)^m \right\}^{k-i} \frac{dz_1}{z_1}\cdots\frac{dz_n}{z_n}=b_{k-i}.
\end{align*}
This completes the proof.
\end{proof}

We call $m_{p, k}(f)$ the {\it $p$-adic $k$-higher Mahler measure} of $f$ and note that $m_{p, 1}(f)=m_p(f)$, which was introduced by Besser and Deninger. The following properties follow from Proposition \ref{exist pHMM}.

\begin{prop}
\begin{enumerate}
\item For any $(l_1, \cdots, l_n) \in \mathbb{Z}^n$ and any positive integer $k$, we have $m_{p, k}(t_1^{l_1}\cdots t_n^{l_n})=0$.
\item Suppose that $f \in \mathbb{C}_p[t_1^{\pm 1}, \dots ,t_n^{\pm 1}]$ does not vanish in any point of $T_p^n$. For any $n \times n$ matrix $S \in M_n(\mathbb{Z}) \cap {\rm GL}_n(\mathbb{Q})$, we have $m_{p, k}(f(t^S))=m_{p, k}(f(t))$.
\end{enumerate}
\end{prop}

\section{Proof of Theorem \ref{main1}}

In this section, we explain Hoffman's algebraic setup of the multiple polylogarithm to make the statement of Theorem \ref{main1} precisely, and prove it.

Let $\mathfrak{H}=\mathbb{Q}\langle e_0, e_1 \rangle$ be the noncommutative polynomial algebra of indeterminates $e_0, e_1$ over $\mathbb{Q}$. We also define its subalgebras $\mathfrak{H}^0$ and $\mathfrak{H}^1$ by
\begin{align*}
\mathfrak{H}^0 =\mathbb{Q}+e_1 \mathfrak{H} e_0 \subset \mathfrak{H}^1 =\mathbb{Q}+e_1 \mathfrak{H}.
\end{align*}
For any positive integer $k$, we put $e_k \coloneqq e_1e_0^{k-1}$. We identify a tuple $(k_1, \cdots, k_r) \in \mathbb{Z}_{>0}^r$ (such a tuple is called an {\it index}) with the monomial $e_{k_1}\cdots e_{k_r} \in \mathfrak{H}^1$.

We define the $\mathbb{Q}$-linear map ${\rm Li} : \mathfrak{H}^1 \rightarrow \mathbb{Q}[[t]]$ by $1 \mapsto 1$ and $e_{k_1}\cdots e_{k_r} \mapsto \operatorname{Li}_{(k_1, \cdots, k_r)}(t)$, where $\operatorname{Li}_{(k_1, \cdots, k_r)}(t)$ is the multipolylogaritm series given by $\eqref{def multipolylog}$. (We note that if $e_{k_1}\cdots e_{k_r} \in \mathfrak{H}^0$, i.e., $k_r \geq 2$, the limit $\lim_{t \to 1} \operatorname{Li}(e_{k_1}\cdots e_{k_r} )$ converges to the multiple zeta value $\zeta(k_1, \cdots, k_r)$ with respect to the Euclidean topology of $\mathbb{R}$. This correspondence gives the $\mathbb{Q}$-linear map $\zeta : \mathfrak{H}^0 \rightarrow \mathbb{R}$.)

\begin{dfn}
We define the $\mathbb{Q}$-bilinear commutative product $\ast : \mathfrak{H}^1 \times \mathfrak{H}^1 \rightarrow \mathfrak{H}^1$ by the bilinearity and the recurrence relations
\begin{enumerate}
\item for each $w \in \mathfrak{H}^1$, $w \ast 1=1 \ast w=w$,
\item for any $k, l \in \mathbb{Z}_{>0}$ and $v, w \in \mathfrak{H}^1$, 
\begin{align*}
e_kv \ast e_lw=e_k(v \ast e_lw)+e_l(e_kv \ast w)+e_{k+l}(v \ast w).
\end{align*}
\end{enumerate}
The product $\ast$ is called the {\it harmonic product} on $\mathfrak{H}^1$.
\end{dfn}

The harmonic product $\ast$ has the following properties.

\begin{prop}[{\cite{Ho97}}]\label{prop harmonic}
\begin{enumerate}
\item For any $w_1, w_2, w_3 \in \mathfrak{H}^1$, we have $w_1 \ast (w_2 \ast w_3)=(w_1 \ast w_2) \ast w_3$ and $w_1 \ast w_2=w_2 \ast w_1$.
\item We equip $\mathfrak{H}^1$ with the multiplication $\ast$. Then $\mathfrak{H}^1_{\ast} \coloneqq (\mathfrak{H}^1, +, \ast)$ forms a commutative $\mathbb{Q}$-algebra.
\item The map ${\rm Li} : \mathfrak{H}^1_{\ast} \rightarrow \mathbb{Q}[[t]]$ is a $\mathbb{Q}$-algebra homomorphism.
\end{enumerate}
\end{prop}

Hoffman also introduced the shuffle product $\shuffle$ on $\mathfrak{H}$, which has similar properties as Proposition \ref{prop harmonic}. (Combining these properties of these two product, we obtain the double shuffle relation of multiple zeta values.)

We also define a variant product to state a preliminary lemma as follows.

\begin{dfn}[{\cite[Section 2]{KY18}}]
We define the $\mathbb{Q}$-bilinear map $\circledast : e_1\mathfrak{H} \times e_1\mathfrak{H} \rightarrow e_1\mathfrak{H}e_0$ by the bilinearity and the recurrence relations
\begin{align*}
ve_k \circledast we_l =(v \ast w)e_{k+l}
\end{align*}
for any $k, l \in \mathbb{Z}_{>0}$ and $v, w \in \mathfrak{H}^1$. We call the product $\circledast$ the {\it circled harmonic product}.
\end{dfn}

We use the following lemma in the proof of Theorem \ref{main1}.

\begin{lem}\label{lem circled}
For any positive integers $k$ and $l$, we have
\begin{align*}
&{\rm Li}_{ (\{1\}^{k-1} \ast \{1\}^{l-1}, 2)} (t)={\rm Li}_{ (\{1\}^{k} \circledast \{1\}^{l})} (t) \\
=&\sum_{\substack{0<m_1<\cdots<m_{k-1}<m \\ 0<n_1<\cdots<n_{l-1}<m}}\frac{t^m}{m_1 \cdots m_{k-1}n_1 \cdots n_{l-1}m^2}.
\end{align*}
\end{lem}

\begin{proof}
Since indices $(\{1\}^{k-1} \ast \{1\}^{l-1}, 2)$ and $(\{1\}^{k} \circledast \{1\}^{l})$ correspond to monomials $(e_1^{k-1} \ast e_1^{l-1})e_2$ and $e_1^k \circledast e_1^l$ respectively, the first equality follows from the definition of the circled harmonic product. The second equality follows from the same argument at the end of Section 2 in \cite{KY18}.
\end{proof}

We now prove Theorem \ref{main1}.

\begin{proof}[Proof of Theorem \ref{main1}]
We compute
\begin{align*}
m_{p, k}((t-\alpha)(t-\beta))
=&\int_{T_p^n} \left\{\log_p (z-\alpha)+\log_p (z-\beta)\right\}^k \frac{dz}{z} \\
=&\sum_{i=0}^k \binom{k}{i} \int_{T_p^n} \left(\log_p^{k-i} (z-\alpha)\right)\left(\log_p^{i} (z-\beta)\right)\frac{dz}{z}.
\end{align*}
and put 
\begin{align*}
I(i) \coloneqq \int_{T_p^n} \left(\log_p^{k-i} (z-\alpha)\right)\left(\log_p^{i} (z-\beta)\right)\frac{dz}{z}
\end{align*}
for $0 \leq i \leq k$. We see that
\begin{align*}
I(0) &= \int_{T_p^n} \log_p^{k} (z-\alpha)\frac{dz}{z}=\int_{T_p^n} \left\{\log_p \alpha+\log_p \left(1-\frac{z}{\alpha}\right)\right\}^k \frac{dz}{z} \\
&=\sum_{j=0}^k \binom{k}{j} \left(\log_p^{k-j} \alpha\right)\int_{T_p^n} \log_p^j \left(1-\frac{z}{\alpha}\right) \frac{dz}{z}=\log_p^{k} \alpha.
\end{align*}
Here, we used the fact that $\log_p^j (1-z/\alpha) \in \mathcal{R}_1(\mathbb{C}_p) \cap z \mathbb{C}_p[[z]]$ for $j \geq 1$ and Lemma \ref{lemBD1} (2) in the last equality. By the similar argument, we also obtain $I(k)=0$. For $1 \leq i \leq k-1$, we have
\begin{align*}
I(i)&=\int_{T_p^n} \left\{\log_p \alpha+\log_p \left(1-\frac{z}{\alpha}\right)\right\}^{k-i}\left\{\log_p z+\log_p \left(1-\frac{\beta}{z}\right)\right\}^{i} \frac{dz}{z} \\
&=\sum_{j=1}^{k-i}\binom{k-i}{j}\left(\log_p^{k-i-j} \alpha\right)\int_{T_p^n} \left(\log_p^j \left(1-\frac{z}{\alpha}\right)\right) \left(\log_p^i \left(1-\frac{\beta}{z}\right)\right) \frac{dz}{z}
\end{align*}
by using Lemma \ref{log int lem} and the fact that the integral for $j=0$ in the summation is zero. Since it follows from \cite[Lemma 1(ii)]{AK99} that $\log_p^j (1-z/\alpha)=(-1)^j j! \operatorname{Li}_{\{1\}^j}(z/\alpha)$ and $\log_p^j (1-\beta/z)=(-1)^i i! \operatorname{Li}_{\{1\}^i}(\beta/z)$, we have
\begin{align*}
I(i)=&\sum_{j=1}^{k-i}(-1)^{i+j}i!j!\binom{k-i}{j}\left(\log_p^{k-i-j} \alpha\right)\times \\ 
&\int_{T_p} \left\{\sum_{0<l_1<\cdots<l_j} \frac{1}{l_1 \cdots l_j} \left(\frac{z}{\alpha}\right)^{l_j} \right\} \left\{\sum_{0<l'_1<\cdots<l'_i} \frac{1}{l'_1 \cdots l'_i} \left(\frac{\beta}{z}\right)^{l'_i} \right\} \frac{dz}{z} \\
=&\sum_{j=1}^{k-i}(-1)^{i+j}i!j!\binom{k-i}{j}\left(\log_p^{k-i-j} \alpha\right)\sum_{\substack{0<l_1<\cdots<l_{j-1}<l \\ 0<l'_1<\cdots<l'_{i-1}<l}}\frac{1}{l_1 \cdots l_{j-1}l'_1 \cdots l'_{i-1}l^2}\left(\frac{\beta}{\alpha}\right)^{l} \\
=&\sum_{j=1}^{k-i}(-1)^{i+j}i!j!\binom{k-i}{j}\left(\log_p^{k-i-j} \alpha\right)\operatorname{Li}_{\left(\{1\}^{i-1}\ast \{1\}^{j-1}, 2\right)} \left(\frac{\beta}{\alpha}\right).
\end{align*}
Here, we used Lemma \ref{lemBD1} (2) in the second equality and Lemma \ref{lem circled} in the last equality. We conclude that
\begin{align*}
&m_{p, k}((t-\alpha)(t-\beta)) \\
=&\log_p^{k} \alpha +\sum_{i=1}^{k-1} \binom{k}{i}\sum_{j=1}^{k-i}(-1)^{i+j}i!j!\binom{k-i}{j}\left(\log_p^{k-i-j} \alpha\right)\operatorname{Li}_{\left(\{1\}^{i-1}\ast \{1\}^{j-1}, 2\right)} \left(\frac{\beta}{\alpha}\right) \\
=&\log_p^k \alpha+\sum_{\substack {i+j \leq k \\ i, j \geq 1}} (-1)^{i+j} \frac{k!}{(k-i-j)!}\left(\log_p^{k-i-j} \alpha\right) \operatorname{Li}_{\left(\{1\}^{i-1}\ast \{1\}^{j-1}, 2\right)} \left(\frac{\beta}{\alpha}\right).
\end{align*}
\end{proof}

\begin{ex}
Let $\alpha, \beta \in \mathbb{C}_p$ with $0<|\beta|<1<|\alpha|$ and $f(t)=(t-\alpha)(t-\beta)$. By Theorem \ref{main1}, we obtain
\begin{align*}
m_{p, 2}(f)&=\log_p^2 \alpha+2\operatorname{Li}_2\left(\frac{\beta}{\alpha}\right), \\
m_{p, 3}(f)&=\log_p^3 \alpha+6(\log_p \alpha)\operatorname{Li}_2\left(\frac{\beta}{\alpha}\right)-12\operatorname{Li}_{(1, 2)}\left(\frac{\beta}{\alpha}\right),
\end{align*}
which are $p$-adic analogues of Remark \ref{Ak comp}.
\end{ex}

At the end of this section, we prove the following proposition, which is a $p$-adic analogue of (\ref{RV}).

\begin{prop}\label{analogueRV}
Let $c \in \mathbb{C}_p$ with $|c|>1$ and $f(t_1, t_2)=t_1+t_1^{-1}+t_2+t_2^{-1}+c$. Then we have
\begin{align*}
m_p(f)=\log_p c -\frac{2}{c^2} {}_4F_3\left(\frac{3}{2}, \frac{3}{2}, 1, 1 ; 2, 2, 2 ; \frac{16}{c^2}\right).
\end{align*}
\end{prop}

\begin{proof}
Put $g(t_1, t_2)=c^{-1}(t_1+t_1^{-1}+t_2+t_2^{-1}) \in \mathfrak{m}_p[t_1^{\pm 1}, t_2^{\pm 1}]$. Then we compute
\begin{align*}
m_p(f)&=\int_{T_p^2} \log_p c\left(1+g(z_1, z_2)\right) \frac{dz_1}{z_1}\frac{dz_2}{z_2} \\
&=\log_p c +\int_{T_p^2} \sum_{m=1}^\infty \frac{(-1)^{m-1}}{m}g(z_1, z_2)^m \frac{dz_1}{z_1}\frac{dz_2}{z_2} \\
&=\log_p c -\sum_{m=1}^\infty  \frac{(-c^{-1})^m}{m} \int_{T_p^2} (z_1+z_1^{-1}+z_2+z_2^{-1})^m \frac{dz_1}{z_1}\frac{dz_2}{z_2} \\
&=\log_p c -\sum_{m=1}^\infty  \frac{(-c^{-1})^m}{m} \int_{T_p^2} \sum_{\substack{i_1+\cdots +i_4=m \\ i_1, \cdots, i_4 \geq 0}} \frac{m!}{i_1!\cdots i_4!}z_1^{i_1-i_2}z_2^{i_3-i_4} \frac{dz_1}{z_1}\frac{dz_2}{z_2} \\
&=\log_p c -\sum_{m=1}^\infty  \frac{c^{-2m}}{2m} \sum_{i=0}^m \frac{(2m)!}{(i!(m-i)!)^2}.
\end{align*}
Here, we used Lemma \ref{lemBD1} (2) in the last equality. Since $(2m)!=4^m\cdot m! \cdot (1/2)_m$ and 
\begin{align*}
\sum_{i=0}^m \binom{m}{i}^2=\binom{2m}{m}=\frac{4^m \cdot \left(\frac{1}{2}\right)_m}{m!},
\end{align*}
we obtain
\begin{align*}
m_p(f)&=\log_p c -\sum_{m=1}^\infty  \frac{c^{-2m}}{2m} \sum_{i=0}^m \frac{4^m\cdot m! \cdot \left(\frac{1}{2}\right)_m}{(i!(m-i)!)^2} \\
&=\log_p c -\sum_{m=1}^\infty \frac{\left(\frac{1}{2}\right)_m}{2m \cdot m!} \left(\frac{4}{c^2}\right)^m \sum_{i=0}^m \binom{m}{i}^2 \\
&=\log_p c -\sum_{m=1}^\infty \frac{\left(\frac{1}{2}\right)_m \left(\frac{1}{2}\right)_m}{(1)_m} \frac{1}{2m \cdot m!} \left(\frac{16}{c^2}\right)^m \\
&=\log_p c -\sum_{m=1}^\infty \frac{\left(\frac{3}{2}\right)_{m-1} \left(\frac{3}{2}\right)_{m-1}}{4 \cdot (2)_{m-1}} \frac{1}{2m \cdot m!} \left(\frac{16}{c^2}\right)^{m-1} \frac{16}{c^2} \\
&=\log_p c -\frac{2}{c^2}\sum_{m=1}^\infty \frac{\left(\frac{3}{2}\right)_{m-1} \left(\frac{3}{2}\right)_{m-1}}{(2)_{m-1}(2)_{m-1}} \frac{(1)_{m-1}}{(2)_{m-1}} \frac{(1)_{m-1}}{(m-1)!} \left(\frac{16}{c^2}\right)^{m-1} \\
&=\log_p c -\frac{2}{c^2} {}_4F_3\left(\frac{3}{2}, \frac{3}{2}, 1, 1 ; 2, 2, 2 ; \frac{16}{c^2}\right).
\end{align*}
\end{proof}

\section{$p$-adic zeta Mahler measures}

In this section, we make some observation on $p$-adic zeta Mahler measures and prove the rest of our main results.

\begin{prop}\label{conv property}
Let $f \in \mathbb{C}_p[t_1^{\pm 1}, \dots ,t_n^{\pm 1}]$ and suppose that $f$ does not vanish in any point of $T_p^n$.
\begin{enumerate}
\item There exists a constant $r>0$ such that $Z_p(X, f)$ is $p$-adically convergent in $D(r^-)$.
\item If there exists a constant $c>1/(p-1)$ such that $|m_{p, k}(f)| \leq p^{-ck}$ for any positive integer $k$, then $Z_p(X, f)$ is $p$-adically convergent in $D(1)$.
\end{enumerate}
\end{prop}

\begin{proof}
(1) We write $f$ as in (\ref{lem 1.2}). Note that, for each $(\zeta_1, \cdots, \zeta_n) \in \mu_{ur}^n$,
\begin{align*}
\left|\log_p \left(1+g(\zeta_1, \cdots, \zeta_n)\right)\right| &= \left|\sum_{m=1}^\infty \frac{(-1)^{m+1}}{m}g(\zeta_1, \cdots, \zeta_n)^m \right| \\
&\leq \sup_{m \geq 1}\left\{|m|^{-1}|g|^m_{\text{Gauss}}\right\} \eqqcolon C_f<\infty,
\end{align*}
where $|g|_{\text{Gauss}}$ is the Gauss norm of $g$. Hence Lemma \ref{lemBD1}(1) implies that
\begin{align*}
|m_{p, k}(f)|\leq \sup_{\zeta \in \mu_{ur}^n} \left\{ \left|\log_p f(\zeta)\right|^k \right\} \leq \left( \max \left\{ |\log_p a|, C_f \right\} \right)^k \leq p^{-Ck}
\end{align*}
for some $C \in \mathbb{R}$ and we obtain 
\begin{align*}
\limsup_{k \to \infty} \left|\frac{m_{p, k}(f)}{k!}\right|^{\frac{1}{k}} \leq p^{\frac{1}{p-1}-C} \eqqcolon r.
\end{align*}
Thus $Z_p(X, f)$ converges in $D(r^-)$. \\
(2) Since we can take $C=c$ and $r>1$ in the proof of the assertion (1), $Z_p(X, f)$ converges in $D(1)$.
\end{proof}

\begin{rmk}
Let $K$ be a finite extension of $\mathbb{Q}_p$ of ramified index $e <p-1$. If $f \in K[t_1^{\pm 1}, \dots ,t_n^{\pm 1}]$ does not vanish in any point of $T_p^n$, then $|m_{p, k}(f)| \leq p^{-k/e}$ for any positive integer $k$. Hence the assumption of Proposition \ref{conv property}(2) is satisfied.
\end{rmk}

\begin{prop}\label{exp rep}
Let $K$ be a finite extension of $\mathbb{Q}_p$ of ramified index $e <p-1$ and suppose that $f \in K[t_1^{\pm 1}, \dots ,t_n^{\pm 1}]$ does not vanish in any point of $T_p^n$. Then, for any $s \in \mathbb{C}_p$ with $|s|\leq 1$, we have
\begin{align*}
Z_p(s, f)=\int_{T_p^n} \exp\left(s \log_p f(z_1, \cdots, z_n) \right) \frac{dz_1}{z_1} \cdots \frac{dz_n}{z_n}.
\end{align*}
\end{prop}

\begin{proof}
For any $s \in \mathbb{C}_p$ with $|s|\leq 1$, the family $\{\sum_{k=0}^m (s\log_p f(z_1, \cdots , z_n))^k/k!\}_{m \geq 0}$ converges uniformly to $\exp\left(s \log_p f(z_1, \cdots, z_n) \right)$ on $\mu_{ur}^n$. Hence the proposition follows from Lemma \ref{ele lem}.
\end{proof}

Finally, we prove Theorem \ref{main2} and Theorem \ref{main3}.

\begin{proof}[Proof of Theorem \ref{main2}]
For any $\zeta \in \mu_{ur}$, we have
\begin{align*}
&\exp(s \log_p (\zeta-\alpha)(\zeta-\beta)) \\
=&\exp\left(s \left(\log_p \alpha+\log_p \zeta+\log_p \left(1-\frac{\zeta}{\alpha}\right)+\log_p \left(1-\frac{\beta}{\zeta}\right) \right) \right) \\
=&\exp\left(s \log_p \langle \alpha \rangle \right) \exp\left(s \log_p \left(1-\frac{\zeta}{\alpha}\right) \right) \exp\left(s \log_p \left(1-\frac{\beta}{\zeta}\right) \right) \\
=&\langle \alpha \rangle^s \left(1-\frac{\zeta}{\alpha}\right)^s \left(1-\frac{\beta}{\zeta}\right)^s.
\end{align*}
Here we note that $|s \log_p \langle \alpha \rangle | \leq |\langle \alpha \rangle-1|<p^{-\frac{1}{p-1}}$ and that $\exp\left(s \log_p \langle \alpha \rangle \right)=\langle \alpha \rangle^s$ (see Lemma 5.5 and Proposition 5.7 in \cite{Wa97}).
Hence, using Proposition \ref{exp rep} and Lemma \ref{lemBD1}(2), we compute
\begin{align*}
Z_p(s, (t-\alpha)(t-\beta))&=\lim_{\substack{N \to \infty \\ (N,p)=1}} \frac{1}{N} \sum_{\zeta \in \mu_N} \exp(s \log_p (\zeta-\alpha)(\zeta-\beta)) \\
&=\lim_{\substack{N \to \infty \\ (N,p)=1}} \frac{1}{N} \sum_{\zeta \in \mu_N} \langle \alpha \rangle^s \left(1-\frac{\zeta}{\alpha}\right)^s \left(1-\frac{\beta}{\zeta}\right)^s \\
&=\langle \alpha \rangle^s \int_{T_p} \left(1-\frac{\zeta}{\alpha}\right)^s \left(1-\frac{\beta}{\zeta}\right)^s \frac{dz}{z} \\
&=\langle \alpha \rangle^s \int_{T_p} \left\{\sum_{n=0}^\infty \binom{s}{n} \left( -\frac{z}{\alpha}\right)^n \right\} \left\{\sum_{n=0}^\infty \binom{s}{n} \left( -\frac{\beta}{z}\right)^n \right\} \frac{dz}{z} \\
&=\langle \alpha \rangle^s \left\{1+ \sum_{n=1}^\infty \binom{s}{n}^2 \left(\frac{\beta}{\alpha}\right)^n \right\} \\
&=\langle \alpha \rangle^s \left\{1+ \sum_{n=1}^\infty \frac{(-s)_n(-s)_n}{(1)_nn!} \left(\frac{\beta}{\alpha}\right)^n \right\} =\langle \alpha \rangle^s {}_2F_1\left(-s, -s ; 1 ; \frac{\beta}{\alpha} \right).
\end{align*}
\end{proof}

\begin{proof}[Proof of Theorem \ref{main3}]
Put $g(t_1, t_2)=c^{-1}(t_1+t_1^{-1}+t_2+t_2^{-1}) \in \mathfrak{m}_p[t_1^{\pm 1}, t_2^{\pm 1}]$. By the similar computation as the proofs of Proposition \ref{analogueRV} and Theorem \ref{main2}, we find that
\begin{align*}
Z_p(s, f)&=\langle c \rangle^s \int_{T_p^2} \left(1+g(z_1, z_2)\right)^s \frac{dz_1}{z_1}\frac{dz_2}{z_2} \\
&=\langle c \rangle^s \int_{T_p^2} \sum_{m=0}^\infty \binom{s}{m} g(z_1, z_2)^m \frac{dz_1}{z_1}\frac{dz_2}{z_2} \\
&=\langle c \rangle^s \sum_{m=0}^\infty \binom{s}{2m} c^{-2m} \sum_{i=0}^m \frac{(2m)!}{(i!(m-i)!)^2} \\
&=\langle c \rangle^s \sum_{m=0}^\infty \frac{\left(-\frac{s}{2}\right)_m\left(\frac{1-s}{2}\right)_m}{(1)_m(1)_m} \left(\frac{4}{c^2}\right)^m \sum_{i=0}^m \binom{m}{i}^2 \\
&=\langle c \rangle^s \sum_{m=0}^\infty \frac{\left(-\frac{s}{2}\right)_m\left(\frac{1-s}{2}\right)_m}{(1)_m(1)_m} \left(\frac{4}{c^2}\right)^m \frac{4^m \cdot \left(\frac{1}{2}\right)_m}{m!} \\
&=\langle c \rangle^s {}_3F_2\left(\frac{1}{2}, -\frac{s}{2}, \frac{1-s}{2} ; 1, 1 ; \frac{16}{c^2} \right).
\end{align*}
\end{proof}

\vspace{10pt}
\noindent
Institute of Mathematics for Industry, Kyushu University,\\
744, Motooka, Nishi-ku, Fukuoka, 819-0395, Japan,\\
E-mail address: \textbf{yu.katagiri.s3@gmail.com}

\end{document}